\documentclass[12pt]{article}
\usepackage{CJK}
\expandafter\let\csname equation*\endcsname\relax
\expandafter\let\csname endequation*\endcsname\relax
\usepackage{amsmath,amssymb,amsthm}
\usepackage{color}
\usepackage{mathtools}
\usepackage{xcolor}
\usepackage{enumerate}
\usepackage{ifpdf}
\usepackage{algorithmic}
\usepackage{subfigure}
\usepackage{float}
\usepackage{setspace}
\usepackage{appendix}
\usepackage{bm}
\usepackage{multirow}

\usepackage{paralist}
\usepackage{booktabs}
\usepackage{array}
\usepackage{tabularx}
\usepackage{subfigure}
\usepackage{amsfonts}
\usepackage{amsmath}\allowdisplaybreaks[4]
\usepackage{fancyhdr}
\usepackage{mathrsfs}

\usepackage[pdftex]{graphicx}
\usepackage{dsfont}
\usepackage[linesnumbered,boxed,ruled,commentsnumbered]{algorithm2e}
\usepackage{caption}

\usepackage{cite}

\usepackage{multirow}
\usepackage{diagbox}
\usepackage{epstopdf}

\restylefloat{figure}

\newcommand{\p}{\partial}
\newcommand{\veps}{\varepsilon}

\newcommand{\wt}{\widetilde}
\newcommand{\no}{\nonumber}

\newcommand{\vu}{\vec{u}}

\newtheorem{assumption}{Assumption}
\newtheorem{theorem}{Theorem}[section]
\newtheorem{lemma}[theorem]{Lemma}

\newtheorem{corollary}{Corollary}[section]

\newtheorem{remark}{Remark}[section]

\newcounter{TableRownNumber}

\newcounter{CIndex}
\newcommand{\nI}{\stepcounter{CIndex}\theCIndex}
\newcommand{\tI}{\theCIndex}

\numberwithin{equation}{section}

\newlength{\FigureHeight}
\newlength{\FigurePageWidth}
\ifpdf
  \usepackage[pdftex]{graphicx}
  \usepackage[pdftex,unicode]{hyperref}

\else
    \usepackage[dvipdfm]{graphicx}
    \usepackage[dvipdfm,unicode]{hyperref}
  \AtBeginDvi{} 
\fi
\hypersetup{backref,
            CJKbookmarks=true,
            bookmarks=true,
            pdfstartview=FitH,
            colorlinks=true,
            bookmarksnumbered=true,
            bookmarksopen=true,
            bookmarksopenlevel=4,
            linkcolor=blue,%
            citecolor=blue,%
            plainpages=false,%
            unicode }
\usepackage{listings}
\begin{document}
\begin{CJK*}{GBK}{}
\title{Carleman estimate for telegrapher's equations on a network and application to an inverse problem}
\author{Yibin Ding
	\thanks{School of Mathematical Sciences, Zhejiang University, Hangzhou, 310027, China. Email: dybmath@yeah.net}
	\and
	Xiang Xu
	\thanks{School of Mathematical Sciences, Zhejiang University, Hangzhou, 310027, China. Email: xxu@zju.edu.cn}
}
  \date{} 
  \maketitle
\begin{abstract}
The main purpose of this work is to study an inverse coefficient problem for telegrapher's equations on a tree-shaped network. To analyze the stability for this inverse problem, Carleman estimate is established first. Based upon this estimate, a Lipschitz stability is further derived with additional data given on partial boundary. 
\end{abstract}

\section{Introduction}
Telegrapher's equations are a pair of coupled linear differential equations which describe the evolution of voltage and current on a transmission line. The equations were originally developed by Oliver Heaviside for centuries where he showed electromagnetic waves could be reflected on wires and wave patterns could appear along the transmission lines. In fact, the equations apply various cases ranging from high frequencies (like radio frequency conductors) to low frequencies (such as power lines) and zero frequency (direct current), see \cite{ME001} and references therein for more applications.   \par
To model transmission lines, it is assumed that lines are composed by an infinite series of an infinitesimal short segment with two-port elementary components which is shown in Figure \ref{fig.01}. \par
\begin{figure}[!ht]\centering
\includegraphics[width=0.7\textwidth]{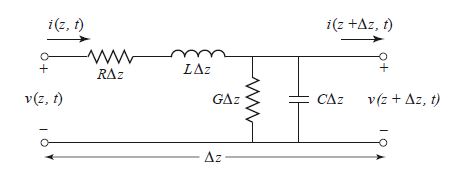}
\caption{Schematic diagram of the elementary components in a transmission line \cite{ME001}. }
\label{fig.01}
\end{figure}
According to Kirchhoff's voltage and current law, the governing equations which are so-called telegrapher's equations can be written as
\begin{gather*}
\begin{cases}
\p_z v(z,t)=-R i(z,t)-L\p_t i(z,t),\\
\p_z i(z,t)=-G v(z,t)-C\p_t v(z,t),
\end{cases}
\end{gather*}
where \(L\), \(C\), \(R\) and \(G\) are distributed inductance (H/m), capacitance (F/m), resistance (\(\Omega\)/m) and conductance (S/m), respectively. Moreover, \(i\) represents current and \(v\) represents voltage. In literature, due to its huge applications in electrical engineering, telegrapher's equations have been intensively investigated. For instance, Sano \cite{Sano001} and Gugat \cite{Gugat001} analyse the stability of this system with different boundary conditions. Allali et al \cite{AllaliAlaaGhammazRouijaa001} develop some new scheme to solve the equations. In addition, Kac \cite{Kac001} discuss a stochastic model of this system and Weiss \cite{Weiss001} show some applications of random walks described by telegrapher's equations. However, most related works on telegrapher's equations focus on forward problems. Albeit important, inverse problems on such kind of equations are still at its infant stage, to the best knowledge of authors. \par
In this paper, we will discuss an inverse coefficient problem for telegrapher's equations on a tree-shaped network. Generally speaking, it does make sense to consider the problem on a tree-shaped network without loops since the signal wire and return wire in transmission lines are typically assumed to be separated. Mathematically speaking, when considering the problem on a tree-shaped network, more or less, we actually investigate a system of telegrapher's equations with transmission boundary conditions imposed on inner vertices. Hence, the strategy is similar as the case of a straight line in which the Carleman estimate plays a significant important role, such as Bellassoued et al \cite{BY001,BY002}, Bukhgeim et al \cite{BK001}, Imanuvilov et al \cite{Imanuvilov001,IY001,IY002,IY003,IY004}, Klibanov et al \cite{Klibanov001,Klibanov2013Carleman}, Klibanov and Timonov \cite{KT001}, and Yamamoto et al \cite{Y001}. Besides, there have been many works on inverse problems for partial differential equations on networks as well. For instance, Baudouin et al \cite{BCV001} obtain a global Carleman estimate on a star-shaped network for the wave equation and apply it to an inverse problem. Ignat et al \cite{IPR001} consider heat equation on a general tree and Schr\"{o}dinger equation on a star-shaped tree, respectively. Baudouin and Yammamoto \cite{BY001} give some uniqueness results for inverse problems on wave, heat and Schr\"{o}dinger equations as well. Three different inverse problems for Schr\"{o}dinger equations are discussed in Avdonin et al \cite{AK001}. Mercier and R\'egnier study the boundary controllability of a network of Euler-Bernoulli beams \cite{MR001}. \par
The rest of this paper is organized as follows. Section \ref{Section.ProblemAndEstimates} is devoted to show the problem and two main results, i.e., Carleman estimate of telegrapher's equations on a tree-shaped network and Lipschitz stability of an inverse coefficient problem. Detailed proofs of the Carleman estimates and stability will be given respectively in the final two sections.
\section{The formulation of governing equations and inverse coefficient problem}\label{Section.ProblemAndEstimates}
Before proceeding to governing equations, we will briefly introduce some notations of a tree-shaped network.
Let \(\Lambda\) be a tree-shaped network consisting of \(N\) edges \((E_j)_{j=1,\dots,N}\) and \(N+1\) vertexes \((V_k)_{k=0,1,\dots,N}\). The length of edge \(E_j\) is \(l_j\). Without losing any generality, we select a vertex that connects only one edge as the root point \(V_0\). Since the tree-shaped network has no loop, there is only one way from \(V_0\) to all points in the network. Thus, the coordinate \(x(A)\) of a point \(A\) in the net can be set as the length of the way from \(V_0\) to \(A\) without any contradiction. Each edge \(E_j\) has two endpoints, which will be set as the initial node \(I_j\) and the terminal node \(T_j\) respectively. And we set \(x(I_j)<x(T_j)\). Thus, we have \(x(I_j)+l_j=x(T_j)\).
For simplicity, we define a characteristic function $\chi_{V_k}$ as follows:
\begin{gather*}
\chi_{V_k}(P)=
\begin{cases}
1,\ P\text{ is }V_k,\\
0,\ P\text{ is not } V_k,
\end{cases}
\end{gather*}
where \(P\) is any point in the net. Set
\begin{align*}
S_{I,k}=&\{j:\chi_{V_k}(I_j)=1,j=1,\dots,N\},\\
S_{T,k}=&\{j:\chi_{V_k}(T_j)=1,j=1,\dots,N\}.
\end{align*}
Thus, \(S_{I,k}\) is the set of index numbers for edges whose initial node is \(V_k\) and \(S_{T,k}\) is the set of index numbers for edges whose terminal node is \(V_k\). Moreover, for simplicity, all vertexes are classified into two separate sets:
\begin{gather*}
\begin{cases}
\displaystyle
\Pi_1=\{k:\sum_{j=1}^N \chi_{V_k}(I_j)+\chi_{V_k}(T_j)=1\},\\
\displaystyle
\Pi_2=\{k:\sum_{j=1}^N \chi_{V_k}(I_j)+\chi_{V_k}(T_j)>1\}.
\end{cases}
\end{gather*}
Here, \(\Pi_1\) is the set of index numbers for boundary vertexes and \(\Pi_2\) is the set of index numbers for all inner vertexes.
Since $\Lambda$ is a tree-shaped network, we have some obvious results as follows:
\begin{gather}
\label{PointSet.Condition}
\begin{cases}
S_{T,0}=\varnothing,\\
 S_{I,k}=\varnothing,\ k\in\Pi_1\backslash\{0\},\\
 |S_{T,k}|=1,\ k=1,2,\dots,N,\\
 |S_{I,k}|\geq 1,\ k\in \Pi_2\cup\{0\}.
 \end{cases}
 \end{gather}
  For instance, $S_{T,0}=\varnothing$ means that no edge ends with root point. $|S_{T,k}|=1$ means that any point is at most the terminal node of one edge.
 \begin{remark}\label{Remark.Network}
 For easiness to readers, we give a simple example of tree-shaped network in Figure \ref{fig.02} which will be utilized in the numerical example in Section \ref{Section.NumericalExperiments} as well. Obviously, according to the definitions of notations, we have \(I_1=V_0,\ T_1=V_1\), \(I_2=V_1,\ T_2=V_2\), \(I_3=V_1,\ T_3=V_3\), \(I_4=V_3,\ T_4=V_4\), \(I_5=V_3,\ T_5=V_5\). Thus, \(\Pi_1=\{0,2,4,5\},\ \Pi_2=\{1,3\}\), \(S_{T,0}=S_{I,2}=S_{I,4}=S_{I,5}=\varnothing\), \(S_{I,0}=\{1\}\), \(S_{I,1}=\{2,3\}\), \(S_{I,3}=\{4,5\}\), \(S_{T,1}=\{1\}\), \(S_{T,2}=\{2\}\), \(S_{T,3}=\{3\}\), \(S_{T,4}=\{4\}\), \(S_{T,5}=\{5\}\). Furthermore, we have \(x(V_0)=0\), \(x(V_1)=l_1\), \(x(V_2)=l_1+l_2\), \(x(V_3)=l_1+l_3\), \(x(V_4)=l_1+l_3+l_4\) and \(x(V_5)=l_1+l_3+l_5\), if \(l_i,\ i=1,2,3,4,5\) are the length of edge \(E_i\) respectively.
 \begin{figure}[!ht]\centering
\includegraphics[width=0.6\textwidth]{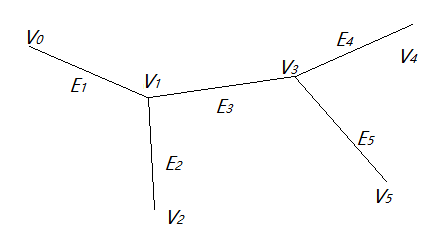}
\caption{An example of tree-shaped network with five edges.}
\label{fig.02}
\end{figure}
 \end{remark}
For convenience, we define an operator \(L(\vec{a}) \):
\begin{gather*}
L(\vec{a})\vec{v}=
\begin{pmatrix}
a_{1} \p_t + a_3&\p_x \\
\p_x& a_2\p_t + a_4
\end{pmatrix}\vec{v},
\end{gather*}
where  \(\vec{a}=(a_1,a_2,a_3,a_4)^T\), \(\vec{v}=(v_1,v_2)^T\). Each components of \(\vec{a}\) represents the distributed inductance, capacitance, resistance, and conductance, respectively. \(v_1\) represents current and \(v_2\) represents voltage.\par
Now, we can give the governing equations with initial and boundary conditions on a network as follows:
\begin{gather}
\label{Electrical.Networks}
\begin{cases}
L(\vec{p}_j)\vu_j=0,\ (x,t)\in Q_j,\ j=1,\dots,N,\\
\vec{u}_j(x,0)=\vec{z}_j(x),\ x\in\mathcal{I}_j,\ j=1,\dots,N,\\
u_{j,2}(x(V_k),t)=\phi_{k}(t),\ k\in\Pi_1,\ j\in S_{I,k}\cup S_{T,k},\ t\in(-T,T),
\end{cases}
\end{gather}
where \(\mathcal{I}_j=(x(I_j),x(T_j))\) and \(Q_j=\mathcal{I}_j\times(-T,T)\). {Moreover, let \(\|p_{j,i}\|_{L^\infty(\mathcal{I})}\leq M_{i},\ j=1,2,\dots,N,\ i=1,2,3,4\) and assume they are consistent on the vertexes:}
\begin{gather}
\label{Coefficient.Assumption}
p_{m,i}=p_{n,i},\ m,n\in S_{T,k}\cup S_{I,k},
k\in\Pi_2,\ i=1,2.
\end{gather}
{ For convenience,denote \(\vec{M}=(M_1,M_2,M_3,M_4)^T\). According to the physical meaning of coefficients, \(p_{j,i}\geq \underline{M}_{i}>0, \ j=1,2,\dots,N,\ i=1,2\) and \(p_{j,i}\geq 0, \ j=1,2,\dots,N,\ i=3,4\). Let \(\{\vec{p}_j\}_{j}\in \mathcal{U}(\vec{M},\underline{M}_1,\underline{M}_2)\), where \(\mathcal{U}(\vec{M},\underline{M}_1,\underline{M}_2)\) is the set of coefficients, whose elements satisfy all conditions above.
}
From Kirchhoff law of current and voltage, we have the following transmission boundary conditions on inner vertexes:
\begin{gather}
\label{KirchhoffLaw}
\begin{cases}
\displaystyle
\sum_{j\in S_{I,k}} u_{j,1}(x(I_j),t)=\sum_{j\in S_{T,k}} u_{j,1}(x(T_j),t),\ k\in\Pi_2,
\\
u_{m,2}(x(V_k),t)=u_{n,2}(x(V_k),t)=:u_2(V_k,t),\ \text{if}\ V_k\in E_m\cap E_n\cap\Pi_2,
\end{cases}t\in[-T,T].
\end{gather}
Moreover, from \eqref{PointSet.Condition}, we know that
\begin{gather*}
|S_{I,k}|\geq |S_{T,k}|=1,\ k\in\Pi_2.
\end{gather*}
Applying Cauchy-Schwartz inequality to \eqref{KirchhoffLaw}, we have
\begin{gather}
\label{Vertex.Inequality}
\sum_{j\in S_{T,k}} {u_{j,1}^2(x(T_j),t)}\leq|S_{I,k}| \sum_{j\in S_{I,k}} {u_{j,1}^2(x(I_j),t)},
k\in\Pi_2.
\end{gather}

{As mentioned in the introduction, the well-posedness of initial boundary value problem of the above system has been obtained by \cite{Sano001,Gugat001}.} However, in this paper, we will mainly investigate the following inverse problem.\\
{\bf Inverse Problem.} Suppose the initial and boundary conditions (\(\vec{z}_j,\  j=1,\dots,N\) and \(\phi_k,\ k\in\Pi_1\) ) are given. Can we stably reconstruct the coefficients \(\vec{p}_j,\ j=1,\dots,N\) in the system with additional boundary observation data, i.e., current on boundary vertexes, $u_{j,1}(x(V_k),t),\ j\in S_{I,k}\cup S_{T,k},\ t\in(-T,T)?$ A partially positive answer will be given in this paper under some assumptions.

Before proceeding to investigate the inverse problem, we need Carleman estimate of the follow system with homogenous initial and boundary conditions:
\begin{gather}
\label{Electrical.Networks.Base}
\begin{cases}
L(\vec{p}_j)\vu_j=\vec{f}_j,\ (x,t)\in Q_j,\ j=1,\dots,N,\\
\vec{u}_j(x,t)=0,\ (x,t)\in\mathcal{I}_j\times\{-T,T\},\ j=1,\dots,N,\\
u_{j,2}(x(V_k),t)=0,\ k\in\Pi_1,\ j\in S_{I,k}\cup S_{T,k},\ t\in(-T,T),
\end{cases}
\end{gather}
which shares the same conditions \eqref{Coefficient.Assumption} and \eqref{KirchhoffLaw} on the inner vertexes as \eqref{Electrical.Networks}.
Set the weight function in the Carleman estimates as
\begin{gather*}
\varphi_j=\alpha_j(x-x_j^*)^2-\beta t^2,\ (x,t)\in Q_j,\ j=1,2,\dots,N.
\end{gather*}
Here, \(x_j^*<x(I_j)\). We can regard \(x_j^*\) as a point out of the edge \(E_j\). Moreover, we can assume
\begin{gather}
\label{weightfunction.assumption}
\begin{cases}
\varphi_i(x(V_k),t)=\varphi_j(x(V_k),t),\ i,j\in S_{I,k}\cup S_{T,k},\\
|S_{I,k}|\p_x \varphi_i(x(V_k),t)=\p_x\varphi_j(x(V_k),t),\ i\in S_{T,k},\ j\in S_{I,k}.
\end{cases}
\end{gather}
It is reasonable since we can easily take $x_j^*=\dfrac{(|S_{I,k}|-1)x(V_k)+x_i^*}{|S_{I,k}|}$ and $\alpha_j=|S_{I,k}|^2\alpha_i$ for $\ i\in S_{T,k},\ j\in S_{I,k}$.
Furthermore, we assume the weight function $\varphi_j$ satisfies the following assumption. It is necessary to guarantee the leading order term to be positive in Carleman estimate in \ref{I_j4}
\begin{assumption}\label{assumption.1}
\begin{gather*}
p_{j,1} p_{j,2}|\p_t\varphi_j|^2-|\p_x\varphi_j|^2\neq0,\ (x,t)\in Q_j,\ j=1,\dots,N.
\end{gather*}
\end{assumption}
Then, we can present the first theorem of this paper, i.e., a Carleman estimate to \eqref{Electrical.Networks.Base}.
\begin{theorem}\label{Theorem.Carleman}
{If \(\vec{u}_j\in L^2(0,T;H^2(\mathcal{I}_j)),\ j=1,2,\dots,N\) is the solution of \eqref{Electrical.Networks.Base},
then there exists \(\tau_0>0\) such that for any \(\tau>\tau_0\), there exists \(s_0=s_0(\tau)>0\), such that
\begin{gather*}
\sum_{j=1}^N s^2\int_{Q_j} |\vec{u}_j|^2 e^{2s\varphi}dxdt \leq
C_{\nI}\sum_{j=1}^N \int_{Q_j} |\vec{f}_j|^2 e^{2s\varphi_j}dxdt+C_{\nI}\wt{B},
\end{gather*}
where
\begin{gather*}
\wt{B}=\sum_{k\in\Pi_1\backslash\{0\}}\sum_{j\in S_{T,k}} \int_{-T}^T s u_{j,1}^2 e^{2s\varphi} dt\Big|_{x=x(T_j)}
\end{gather*}
for all \(s> s_0\). Here \(C\) is a constant, which is dependent on \(\tau_0\), \(s_0\), \(\{\mathcal{I}_j\}\), \(T\), \(\vec{M}\), \(\underline{M}_{1}\), \(\underline{M}_{2}\) and independent of \(\{\vec{p}_j\}\), \(s\).
}
\end{theorem}
The detailed proof of this theorem will be given in next section. Based on this estimate, we could give the second theorem on stability analysis of the inverse problem.  Suppose we have two systems with two independent set of measurement data, \(m=1,2\):
\begin{gather}
\label{ip.sys.01}
\begin{cases}
L(\vec{p}_j)\vec{u}_{m,j}=0,\ (x,t)\in Q_j,\\
\vec{u}_{m,j}(x,0)=\vec{z}_{m,j}(x),\ x\in(x(I_j),x(T_j)),j=1,\dots,N,\\
\vec{u}_{m,j}(x(V_k),t)=\phi_{m,j}(t),\ j\in S_{I,k}\cup S_{T,k},\ k\in \Pi_1.
\end{cases}
\end{gather}
\begin{gather}
\label{ip.sys.02}
\begin{cases}
L(\vec{q}_{j})\vec{v}_{m,j}=0,\ (x,t)\in Q_j,\\
\vec{v}_{m,j}(x,0)=\vec{z}_{m,j}(x),\ x\in(x(I_j),x(T_j)),j=1,\dots,N,\\
\vec{v}_{m,j}(x(V_k),t)=\phi_{m,j}(t),\ j\in S_{I,k}\cup S_{T,k},\ k\in \Pi_1.
\end{cases}
\end{gather}

Here, \(\vec{p}\) and \(\vec{q}\) are two set of independent coefficients. Denote \(\vec{\rho}_{j}=\vec{q}_{j}-\vec{p}_{j}\).
Moreover, we assume that
\begin{assumption}\label{Inverse.Assumption}
We select \(z_{1,j,1},z_{1,j,2},z_{2,j,1},z_{2,j,2}\) to satisfy that
\begin{gather*}
\det
\begin{pmatrix}
\p_x z_{1,j,2}&0& z_{1,j,1} &0\\
0&\p_x z_{1,j,1}&0&z_{1,j,2}\\
\p_x z_{2,j,2}&0& z_{2,j,1} &0\\
0&\p_x z_{2,j,1}&0&z_{2,j,2}
\end{pmatrix}\neq 0.
\end{gather*}
\end{assumption}
Then we can get the Lipschitz stability of this inverse problem:
\begin{theorem}\label{IP.Theorem}
Under the Assumption \ref{Inverse.Assumption}, and assume that \(\vec{w}_{m,j}\in (H^3(Q_j))^2\), then there exists a constant \(C\), which is dependent on \(\{\mathcal{I}_j\}\), \(T\), \(\vec{M}\), \(\underline{M}_{1}\), \(\underline{M}_{2}\) and independent of \(\{\vec{p}_j\}\), \(s\), such that
\begin{gather*}
\sum_{j=1}^{N}\sum_{n=1}^4 \int_{\mathcal{I}_j}|\rho_{j,n}|^2dx
\leq
C
\sum_{l,m=1}^2\sum_{k\in\Pi_1\backslash\{0\}}\sum_{j\in S_{T,k}} \int_{-T}^T  |\p_t^l w_{m,j,1}|^2  dt\Big|_{x=x(T_j)}.
\end{gather*}

\end{theorem}

\section{Proof of Main results}
In this section, we will give detailed proof to Theorem \ref{Theorem.Carleman} and Theorem \ref{IP.Theorem}. The methodology is similar as most papers on Carleman estimate such as Imanuvilov et al\cite{Imanuvilov001,IY001,IY002,IY003,IY004}.

\begin{proof}[Proof of Theorem \ref{Theorem.Carleman}]
Set \(\vec{w}_j=\vec{u}_j e^{s\varphi_j}\) and denote
\begin{gather*}
(L \vec{u} )\cdot e^{s\varphi_j}=
\begin{pmatrix}
p_{j,1} \p_t w_{j,1}+\p_x w_{j,2}-s(\p_t\varphi_j)p_{j,1} w_{j,1}-s(\p_x\varphi_j)w_{j,2}+p_{j,3} w_{j,1}\\
\p_x w_{j,1}+p_{j,2}\p_t w_{j,2}-s(\p_x\varphi_j)w_{j,1}-s(\p_t\varphi_j)p_{j,2} w_{j,2}+p_{j,4} w_{j,2}
\end{pmatrix}.
\end{gather*}
Take
\(
\wt{L}\vec{w}=(L \vec{u} )\cdot e^{s\varphi_j}-
\begin{pmatrix}
p_{j,3} w_{j,1}\\
p_{j,4} w_{j,2}
\end{pmatrix}
\). Then, by straightforward calculation we have
\begin{align*}
&\int_{Q_j}
(\wt{L}\vec{w}_j)^{T}
\begin{pmatrix}
p_{j,2}&0\\
0&p_{j,1}
\end{pmatrix}
\wt{L}\vec{w}_j
dxdt
\\
=
&
\int_{Q_j} p_{j,2}(p_{j,1}\p_t w_{j,1}+\p_x w_{j,2})^2+p_{j,1}(\p_x w_{j,1}+p_{j,2}\p_t w_{j,2})^2 dxdt
\\
&
-\int_{Q_j} 2s p_{j,2}(p_{j,1}\p_t w_{j,1}+\p_x w_{j,2})(p_{j,1}\p_t\varphi_j\cdot w_{j,1}+\p_x\varphi_j\cdot w_{j,2})dxdt
\\
&
-\int_{Q_j} 2s p_{j,1}(\p_x w_{j,1}+p_{j,2}\p_t w_{j,2})(\p_x\varphi_j\cdot w_{j,1}+p_{j,2}\p_t\varphi_j\cdot w_{j,2})dxdt
\\
&
+\int_{Q_j} s^2[p_{j,2}(p_{j,1}\p_t\varphi_j\cdot w_{j,1}+\p_x\varphi_j\cdot w_{j,2})^2+p_{j,1}(\p_x\varphi_j\cdot w_{j,1}+p_{j,2}\p_t\varphi_j\cdot w_{j,2})^2] dxdt
\\
\triangleq
&
\sum_{i=1}^4 I_{j,i}.
\end{align*}
Obviously, we have
\begin{align*}
I_{j,1}\geq& 0
\end{align*}
\begin{align*}
I_{j,2}+I_{j,3}=
&
-\int_{Q_j} 2s p_{j,2}(p_{j,1}\p_t w_{j,1}+\p_x w_{j,2})(p_{j,1}\p_t\varphi_j\cdot w_{j,1}+\p_x\varphi_j\cdot w_{j,2})dxdt
\no\\
&
-\int_{Q_j} 2s p_{j,1}(\p_x w_{j,1}+p_{j,2}\p_t w_{j,2})(\p_x\varphi_j\cdot w_{j,1}+p_{j,2}\p_t\varphi_j\cdot w_{j,2})dxdt\no\\
=&\int_Q s[p_{j,1}^2 p_{j,2} \p_t^{2}\varphi_j+\p_x(p_{j,1}\p_x\varphi_j)] w_1^2+s[p_{j,1} p_{j,2}^2 \p_t^{2}\varphi_j+\p_x(p_{j,2}\p_x\varphi_j)] w_2^2dxdt
\no\\
&+\int_Q s\{4p_{j,1} p_{j,2}\p_t\p_x\varphi_j+2[\p_x(p_{j,1} p_{j,2})]\p_t\varphi_j\} w_1 w_2 dxdt
\no
\\
&-\int_Q \p_x[s\p_x\varphi_j(p_{j,1}w_1^2+p_{j,2}w_2^2)+2sp_{j,1}p_{j,2}(\p_t\varphi_j)w_1 w_2] dxdt
\end{align*}
\begin{align*}
I_{j,4}=
&\int_{Q_j} s^2[p_{j,2}(p_{j,1}\p_t\varphi_j\cdot w_{j,1}+\p_x\varphi_j\cdot w_{j,2})^2+p_{j,1}(\p_x\varphi_j\cdot w_{j,1}+p_{j,2}\p_t\varphi_j\cdot w_{j,2})^2] dxdt
\no\\
=&\int_{Q_j} s^2 \Big[(p_{j,1} p_{j,2}|\p_t\varphi_j|^2+|\p_x\varphi_j|^2)(p_{j,1}w_{j,1}^2+p_{j,2}
w_{j,2}^2)
\no\\&
\hphantom{\int_{Q_j} s^2 \Big[}
+4p_{j,1} p_{j,2}(\p_t\varphi_j)(\p_x\varphi_j) w_{j,1} w_{j,2})\Big] dxdt.
\end{align*}
From Assumption \ref{assumption.1}, we have
\begin{align*}
&
16 p_{j,1}^2 p_{j,2}^2 |\p_t\varphi_j|^2|\p_x\varphi_j|^2
-4p_{j,1} p_{j,2}(p_{j,1} p_{j,2}|\p_t\varphi_j|^2+|\p_x\varphi_j|^2)^2
\\
=&-4p_{j,1} p_{j,2}(p_{j,1} p_{j,2}|\p_t\varphi_j|^2-|\p_x\varphi_j|^2)^2<0,
\end{align*}
thus,
\begin{gather}\label{I_j4}
 C_{\nI} s^2\int_{Q_j} \Big(w_{j,1}^2+w_{j,2}^2 \Big)dxdt\leq I_{j,4}.
\end{gather}
Moreover, by taking a sufficiently large \(s\), we can absorb all lower order terms of $s$ in $I_{j,2} + I_{j,3}$ and obtain
\begin{gather}
\label{ElectricalNetworks.Edge.Inequality}
C_{\nI} s^2\int_{Q_j} \Big(w_{j,1}^2+w_{j,2}^2 \Big) dxdt \leq \int_{Q_j}
(\wt{L}\vec{w}_j)^{T}
\begin{pmatrix}
p_{j,2}&0\\
0&p_{j,1}
\end{pmatrix}
\wt{L}\vec{w}_j
dxdt
+B_{j},
\end{gather}
where $B_j$ are related to boundary terms, i.e.,
\begin{gather*}
B_j=\int_{Q_j} \p_x[s\p_x\varphi_j(p_{j,1}w_{j,1}^2+p_{j,2}w_{j,2}^2)+2sp_{j,1}p_{j,2}(\p_t\varphi_j)w_{j,1} w_{j,2}] dxdt
\end{gather*}
Since we add all segments at last, we actually need to estimate \(B=\sum_{j=1}^{N} B_j\) as follow,
\begin{align*}
B=&\sum_{j=1}^{N} B_j=
\sum_{j=1}^{N} \int_{-T}^T [s\p_x\varphi_j(p_{j,1}u_{j,1}^2+p_{j,2}u_{j,2}^2)+2sp_{j,1}p_{j,2}(\p_t\varphi_j)u_{j,1} u_{j,2} ]e^{2s\varphi_j} dt\Big|_{x(I_j)}^{x(T_j)}
\no\\
\triangleq&
\sum_{k=0}^N D_k,
\end{align*}
where
\begin{align*}
D_k=&
\sum_{j\in S_{T,k}}\int_{-T}^T s\p_x\varphi_j(p_{j,1}u_{j,1}^2+p_{j,2}u_{j,2}^2)e^{2s\varphi_j} dt\Big|_{x=x(T_j)}
\no\\
&
-\sum_{j\in S_{I,k}}\int_{-T}^T s\p_x\varphi_j(p_{j,1}u_{j,1}^2+p_{j,2}u_{j,2}^2)e^{2s\varphi_j} dt\Big|_{x=x(I_j)}
\no\\
&
+\sum_{j\in S_{T,k}}\int_{-T}^T 2sp_{j,1}p_{j,2}(\p_t\varphi_j)u_{j,1} u_{j,2} e^{2s\varphi_j} dt\Big|_{x=x(T_j)}
\no\\
&
-\sum_{j\in S_{I,k}}\int_{-T}^T 2sp_{j,1}p_{j,2}(\p_t\varphi_j)u_{j,1} u_{j,2} e^{2s\varphi_j} dt\Big|_{x=x(I_j)}.
\end{align*}
If \(k=0\),
\begin{gather*}
D_0=-\sum_{j\in S_{I,0}}\int_{-T}^T s\p_x\varphi_j(p_{j,1}u_{j,1}^2+p_{j,2}u_{j,2}^2)e^{2s\varphi_j} dt\Big|_{x=x(I_j)}\leq 0.
\end{gather*}
If \(k\in\Pi_1\backslash\{0\}\), from \eqref{Electrical.Networks.Base}, we have
\begin{align*}
D_k=&
\sum_{j\in S_{T,k}}\int_{-T}^T s\p_x\varphi_j(p_{j,1}u_{j,1}^2+p_{j,2}u_{j,2}^2)e^{2s\varphi_j} dt\Big|_{x=x(T_j)}
\\
\leq
&C_{\nI} \sum_{j\in S_{T,k}}\int_{-T}^T s u_{j,1}^2 e^{2s\varphi_j} dt\Big|_{x=x(T_j)}.
\end{align*}
Next, we discuss the case of \(k\in\Pi_2\). From \eqref{Coefficient.Assumption} \eqref{Vertex.Inequality} \eqref{weightfunction.assumption}, we have
\begin{gather*}
\sum_{j\in S_{T,k}}\int_{-T}^T s\p_x\varphi_j p_{j,1}u_{j,1}^2 e^{2s\varphi_j} dt\Big|_{x=x(T_j)}
\leq
\sum_{j\in S_{I,k}}\int_{-T}^T s\p_x\varphi_j p_{j,1}u_{j,1}^2e^{2s\varphi_j} dt\Big|_{x=x(I_j)},\\
\sum_{j\in S_{T,k}}\int_{-T}^T s\p_x\varphi_j p_{j,2}u_{j,2}^2 e^{2s\varphi_j} dt\Big|_{x=x(T_j)}
=
\sum_{j\in S_{I,k}}\int_{-T}^T s\p_x\varphi_j p_{j,2}u_{j,2}^2e^{2s\varphi_j} dt\Big|_{x=x(I_j)},
\end{gather*}
\begin{align*}
&\sum_{j\in S_{T,k}}\int_{-T}^T 2 s p_{j,1}p_{j,2}(\p_t\varphi_j)u_{j,1} u_{j,2} e^{2s\varphi_j} dt\Big|_{x=x(T_j)}\\
=&
\sum_{j\in S_{I,k}}\int_{-T}^T 2 s p_{j,1}p_{j,2}(\p_t\varphi_j)u_{j,1} u_{j,2} e^{2s\varphi_j} dt\Big|_{x=x(I_j)}.
\end{align*}
which implies $D_k\leq 0,\ k\in\Pi_2$.

Combining all estimates of \(D_k\) gives the estimate of \(B\):
\begin{gather*}
B\leq \sum_{k\in\Pi_1\backslash\{0\}}D_k
\leq
C_{\nI}\wt{B},
\end{gather*}
where
\begin{gather*}
\wt{B}=\sum_{k\in\Pi_1\backslash\{0\}} \sum_{j\in S_{T,k}} \int_{-T}^T s u_{j,1}^2 e^{2s\varphi_j} dt\Big|_{x=x(T_j)}.
\end{gather*}
Finally according to \eqref{ElectricalNetworks.Edge.Inequality}, the estimate of \(B\) and the definition of \(\wt{L}\vec{w}\), we can complete the proof and obtain
\begin{gather*}
\sum_{j=1}^N s^2\int_{Q_j} |\vec{u}_j|^2 e^{2s\varphi_j}dxdt \leq
C_{\nI}\sum_{j=1}^N \int_{Q_j} |L \vec{u}_j|^2 e^{2s\varphi_j}dxdt+C_{\nI}\wt{B}.
\end{gather*}

Once we obtain the Carleman estimate of telegrapher's equations on the tree-shaped network, we can go further to prove the stability of inverse coefficient problem.
\end{proof}\par
{

Before proofing Theorem \ref{IP.Theorem}, the energy estimate of system \eqref{Electrical.Networks} and \eqref{Electrical.Networks.Base} will be derived.
\begin{lemma}\label{energy.estimate.01}
If \(\vec{u}_j\) is the solution of \eqref{Electrical.Networks}, then there exists a constant \(C\), which is dependent on \(\Omega\), \(t\), \(\vec{M}\) and independent of \(\{\vec{p}_j\}\), such that
\begin{gather*}
\sum_{j=1}^N \int_{\mathcal{I}_j} u_{j,1}^2(t)+u_{j,2}^2(t) dx \leq  C\left[\sum_{j=1}^N\int_{\mathcal{I}_j} z_{j,1}^2(x)+z_{j,2}^2(x) dx+ \sum_{j\in\Pi_1} \int_{0}^t \phi_k^2(\tau) d\tau\right].
\end{gather*}
\end{lemma}
\begin{lemma}\label{energy.estimate.02}
If \(\vec{u}_j\) is the solution of \eqref{Electrical.Networks.Base}, then there exists a constant \(C\), which is dependent on \(\Omega\), \(t\), \(\vec{M}\) and independent of \(\{\vec{p}_j\}\), such that
\begin{gather*}
\sum_{j=1}^N \int_{\mathcal{I}_j} u_{j,1}^2(t)+u_{j,2}^2(t) dx \leq  C \sum_{j=1}^N (\sum_{j=1}^N \int_{\mathcal{I}_j} u_{j,1}^2(0)+u_{j,2}^2(0) dx+\int_{Q_k(t)} f_{j,1}^2+f_{j,2}^2 dxdt),
\end{gather*}
where \(Q_j(t)=\mathcal{I}_j\times(0,t)\).
\end{lemma}
The proofs of these two lemma are similar with basic energy estimates, so we only give some key steps.
Lemma \ref{energy.estimate.01} is based on Lemma \ref{energy.estimate.02} and homogenization of boundary conditions.
To proof Lemma \ref{energy.estimate.02}, \(\sum_{j=1}^N \int_{Q_j(t)}\vec{f}_j^{~T} \vec{u}_j dxdt\) will be computed.
Transmission boundary conditions \eqref{KirchhoffLaw} on inner vertexes and boundary conditions in \eqref{Electrical.Networks.Base} are used to derive \(\sum_{j=1}^N\int_{Q_j(t)}\p_x(u_{j,1}u_{j,2}) dxdt=0\), which will attend in \(\sum_{j=1}^N \int_{Q_j(t)}\vec{f}_j^{~T} \vec{u}_j dxdt\).\par
Based on Lemma \ref{energy.estimate.01}, the follow corollary can be derived:
\begin{corollary}\label{energy.estimate.03}
There exists a constant \(\hat{M}>0\), which is dependent on \(\Omega\), \(T\), \(\vec{M}\), \(\underline{M}_{1}\), \(\underline{M}_{2}\), \(\{\vec{z}_{j}\}\), \(\{\phi_k\}_{k\in\Pi_1}\), such that
\begin{gather*}
\sum_{j=1}^N \sum_{k=0}^3 \int_{\mathcal{I}_j} |\p_t^k u_{j,1}(t)|^2+|\p_t^k u_{j,2}(t)|^2 dx \leq \hat{M}
\end{gather*}
for all \(t\in[0,T]\).
\end{corollary}
}
\begin{proof}[Proof of Theorem \ref{IP.Theorem}]
By the definition of the function \(\varphi_j\), {there exists a \(d_1>0\) and a sufficient large \(T\), such that}
\begin{gather*}
\varphi_j(x,0){\geq} d_{1},\ \varphi_j(x,-T)=\varphi_j(x,T)<d_{1},\ x\in\mathcal{I}_j,\ j=1,\dots,N.
\end{gather*}
Thus, for a given \(\varepsilon>0\), we can choose a sufficiently small \(\delta>0\), such that
\begin{gather*}
\varphi_j(x,t)\geq d_{1}-\varepsilon,\ (x,t)\in\mathcal{I}_j\times[-\delta,\delta],\ j=1,2,\dots,N
\end{gather*}
and
\begin{gather*}
\varphi_j(x,t)\leq d_{1}-2\varepsilon,\ (x,t)\in\mathcal{I}_j\times[-T,-T+2\delta]\cup[T-2\delta,T],\ j=1,2,\dots,N.
\end{gather*}
To apply Carleman estimates, it is necessary to introduce a cut-off function \(\chi\) satisfying \(0\leq \chi\leq 1,\ \chi\in C^{\infty}(\mathbb{R})\) and
\begin{gather*}
\chi(t)=
\begin{cases}
0,\ t\in[-T,-T+\delta]\cup[T-\delta,T],\\
1,\ t\in[-T+2\delta,T-2\delta].
\end{cases}
\end{gather*}
{
Now, let us consider the difference between \eqref{ip.sys.01} and \eqref{ip.sys.02}
Denote \(\vec{w}_{m,j}=(\vec{u}_{m,j}-\vec{v}_{m,j})\), thus
\begin{gather}\label{system.delta}
\begin{cases}
L(\vec{p}_{j})\vec{w}_{m,j}=f_{m,j},\ (x,t)\in Q_j,\ j=1,\dots,N,\\
\vec{w}_{m,j}(x,0)=0,\ x\in(x(I_j),x(T_j)),\ j=1,\dots,N,\\
\vec{w}_{m,j}(x(V_k),t)=0,\ j\in S_{I,k}\cup S_{T,k},\ k\in \Pi_1,
\end{cases}
\end{gather}
where
\begin{gather*}
f_{m,j}=
\begin{pmatrix}
\rho_{j,1}\p_t v_{m,j,1}+\rho_{j,3} v_{m,j,1}\\
\rho_{j,2}\p_t v_{m,j,2}+\rho_{j,4} v_{m,j,2}
\end{pmatrix}.
\end{gather*}
}
Obviously, at \(t=0\) we have
\begin{gather*}
\begin{cases}
p_{j,1}\p_t w_{m,j,1}=-\rho_{j,1}\dfrac{\p_x z_{1,j,2}+q_{j,3}z_{1,j,1}}{q_{j,1}}+\rho_{j,3} z_{m,j,1},\\
p_{j,2}\p_t w_{m,j,2}=-\rho_{j,2}\dfrac{\p_x z_{1,j,1}+q_{j,4}z_{1,j,2}}{q_{j,2}}+\rho_{j,4} z_{m,j,2}.
\end{cases}
\end{gather*}
According to Assumption \ref{Inverse.Assumption}, we have
\begin{gather}\label{rho.e1}
\sum_{n=1}^4 \int_{\mathcal{I}_j}|\rho_{j,n}|^2 dx
\leq C_{\nI}\int_{\mathcal{I}_j} \sum_{m=1}^2(|\p_t w_{m,j,1}|^2+|\p_t w_{m,j,2}|^2)dx\Big|_{t=0}.
\end{gather}
Here,
\begin{align*}
&
 \int_{\mathcal{I}_j} (|\p_t w_{m,j,1}|^2+|\p_t w_{m,j,2}|^2)e^{2s\varphi_j} dx\Big|_{t=0}
 \\
 =&\int_{-T}^0 \frac{\p}{\p t}\int_{\mathcal{I}_j} \chi^2 (|\p_t w_{m,j,1}|^2+|\p_t w_{m,j,2}|^2)e^{2s\varphi_j} dx dt
 \\
 \leq &C_{\nI}\int_{Q_j} s\chi^2(|\p_t w_{m,j,1}|^2+|\p_t w_{m,j,2}|^2+|\p_t^2 w_{m,j,1}|^2+|\p_t^2 w_{m,j,2}|^2)e^{2s\varphi_j} dxdt
 \\
 &+\int_{-T}^0\int_{\mathcal{I}_j} (|\p_t w_{m,j,1}|^2+|\p_t w_{m,j,2}|^2)e^{2s\varphi_j}(\p_t(\chi^2)) dxdt
 \\
 \leq &C_{\nI}\int_{Q_j} s\chi^2(|\p_t w_{m,j,1}|^2+|\p_t w_{m,j,2}|^2+|\p_t^2 w_{m,j,1}|^2+|\p_t^2 w_{m,j,2}|^2)e^{2s\varphi_j} dxdt
 \\
 &+C_{\tI} e^{2s(d_1-2\varepsilon)}{\int_{-T+\delta}^{-T+2\delta}\int_{\mathcal{I}_j}(|\p_t w_{m,j,1}|^2+|\p_t w_{m,j,2}|^2)dxdt}
 .
\end{align*}
In order to estimate those terms of derivatives with respect to $t$ above, we have to take $t$-derivative to the original telegrapher's equations and reuse the Carleman estimate again. Denote by \(y_{l,m,j}=\chi\p_t^l \vec{w}_{m,j},\ l=1,2\). Then we have the following systems
\begin{gather*}
\begin{cases}
L(\vec{p}_{j})\vec{y}_{l,m,j}=\chi \p_t^l
\vec{f}_{m,j}-\p_t\chi \p_t^l \vec{w}_{m,j},\ (x,t)\in Q_j,\ j=1,\dots,N,\\
\vec{y}_{l,m,j}(x,t)=0,\ (x,t)\in(x(I_j),x(T_j))\times\{0,\pm T\},\ j=1,\dots,N,\\
\vec{y}_{l,m,j}(x(V_k),t)=0,\ j\in S_{I,k}\cup S_{T,k},\ k\in \Pi_1.
\end{cases}
\end{gather*}
Applying Theorem \ref{Theorem.Carleman}, we have
\begin{align*}
&\sum_{j=1}^N s^2\int_{Q_j} (|\chi\p_t^l w_{m,j,1}|^2+|\chi\p_t^l w_{m,j,2}|^2) e^{2s\varphi_j}dxdt
\\
\leq&
C_{\nI}\sum_{j=1}^N( \int_{Q_j} |\chi\p_t^l \vec{f}_{m,j}|^2 e^{2s\varphi_j}+ |\p_t\chi \p_t^l \vec{w}_{m,j}|^2 e^{2s\varphi_j} dxdt)+C_{\nI}s\wt{B}_{l,m}\\
\leq&
C_{\nI}\sum_{j=1}^N\int_{Q_j}|\p_t^l \vec{f}_{m,j}|^2 e^{2s\varphi_j}dxdt\\
&+C_{\nI}\sum_{j=1}^N (\int_{-T+\delta}^{-T+2\delta}+\int_{T-2\delta}^{T-\delta})\int_{\mathcal{I}_j}|\p_t^l \vec{w}_{m,j}|^2dxdt e^{2s(d_{1}-2\veps)}dxdt+C_{\nI}s\wt{B}_{l,m}\\
\triangleq& C_{\nI}(P_{1,l}+P_{2,l}+s\wt{B}_{l,m})
\end{align*}
where
\begin{gather*}
\wt{B}_{l,m}=\sum_{k\in\Pi_1\backslash\{0\}}\sum_{j\in S_{T,k}} \int_{-T}^T  \chi^2 |\p_t^l w_{m,j,1}|^2 e^{2s\varphi_j} dt\Big|_{x=x(T_j)}.
\end{gather*}
Corollary \ref{energy.estimate.03} gives the estimate of \(P_{1,l}\):
\begin{gather*}
P_{1,l}=\sum_{j=1}^N\int_{Q_j}|\p_t^l\vec{f}_{m,j}|^2 e^{2s\varphi_j}dxdt
\leq
C_{\nI}\hat{M}T\sum_{j=1}^N\sum_{n=1}^4\int_{Q_j}\rho_{j,n}^2 e^{2s\varphi_j}dxdt.
\end{gather*}
Applying Lemma \ref{energy.estimate.02} on the system which \(\p_t^l \vec{w}_{m,j},\ l=1,2\) satisfy, estimates of \(P_{2,l}\) can be derived:
\begin{align*}
&P_{2,l}=\sum_{j=1}^N(\int_{-T+\delta}^{-T+2\delta}+\int_{T-2\delta}^{T-\delta})\int_{\mathcal{I}_j}(|\p_t^l w_{m,j,1}|^2+|\p_t^l w_{m,j,2}|^2)dxdt\\
\leq & C_{\nI} \sum_{j=1}^N\sum_{n=1}^4\int_{\mathcal{I}_j}\rho_{j,n}^2 dx.
\end{align*}
Thus, do the summation of  \eqref{rho.e1} with respect to $j$ gives
\begin{align*}
&\sum_{j=1}^{N}\sum_{n=1}^4 s\int_{\mathcal{I}_j} |\rho_{j,n}|^2e^{2s\varphi_j}dx\Big|_{t=0}
\\
\leq& C_{\nI}\sum_{j=1}^N\sum_{n=1}^4\int_{Q_j}\rho_{j,n}^2 e^{2s\varphi_j}dxdt
+C_{\tI}s\sum_{m,l=1}^2 \wt{B}_{l,m}\\
&+C_{\tI}se^{2s(d_1-2\varepsilon)}\sum_{j=1}^N\int_{-T+\delta}^{-T+2\delta}\int_{\mathcal{I}_j}(|\p_t w_{m,j,1}|^2+|\p_t w_{m,j,2}|^2)dxdt.
\end{align*}
Similar with \(P_{2,j}\), the estimate of the last term of above inequality can be derived and at the beginning of this proof, we have known that \(\varphi_j(x,0)\geq d_1\). Hence, taking \(s>0\) large,
\begin{gather*}
\sum_{j=1}^{N}\sum_{n=1}^4 \int_{\mathcal{I}_j} |\rho_{j,n}|^2 e^{2s\varphi_j(\cdot,0)}dx\leq
C_{\nI} \sum_{m,l=1}^2 \sum_{k\in\Pi_1\backslash\{0\}}\sum_{j\in S_{T,k}} \int_{-T}^T  |\p_t^l w_{m,j,1}|^2 e^{2s\varphi_j} dt\Big|_{x=x(T_j)}.
\end{gather*}
Thus we complete the proof of Theorem \ref{IP.Theorem}.
\end{proof}

\bibliographystyle{unsrt}
\bibliography{ref}
\nocite{*}

\end{CJK*}

\end{document}